\documentclass[a4paper,11pt,leqno]{amsart}
\usepackage{amsmath,amsthm,amssymb,amsfonts,graphicx,latexsym,mathrsfs,subfig,bm,hyperref}
\usepackage[all]{xy}
\numberwithin{equation}{section}
\theoremstyle{plain}
\newtheorem{thm}{Theorem}[section]
\newtheorem{lem}[thm]{Lemma}
\newtheorem{prop}[thm]{Proposition}
\newtheorem{cor}[thm]{Corollary}
\newtheorem{rem}[thm]{Remark}

\theoremstyle{definition}

\def\R{\mathbb{R}}
\def\N{\mathbb{N}}

\def\e{\epsilon}
\def\g{\gamma}
\def\d{\delta}
\def\D{\Delta}

\def\G{\Gamma}
\def\a{\alpha}
\def\P{\mathcal{P}}

\def\k{\kappa}

\DeclareMathOperator{\diam}{diam}
\DeclareMathOperator{\dist}{dist}

\begin{document}

\title[Bi-Lipschitz embedding of the generalized Grushin plane]{Bi-Lipschitz embedding of the generalized Grushin plane into Euclidean spaces}
\date{\today}
\author{Matthew Romney}
\address{Department of Mathematics , University of Illinois at Urbana-Champaign, Urbana, IL 61801, USA}
\email{romney2@illinois.edu}
\author{Vyron Vellis}
\address{Department of Mathematics and Statistics, P.O. Box 35 (MaD), FI-40014 University of Jyv\"askyl\"a, Jyv\"askyl\"a, Finland}
\email{vyron.v.vellis@jyu.fi}
\thanks{The second author was supported in part by the Academy of Finland project 257482.}
\subjclass[2010]{Primary 53C17; Secondary 30L05, 30L10}
\keywords{generalized Grushin plane, quasiplane, bi-Lipschitz embedding}

\begin{abstract}
We show that, for all $\a\geq 0$, the generalized Grushin plane $\mathbb{G}_{\a}$ is bi-Lipschitz homeomorphic to a $2$-dimensional quasiplane in the Euclidean space $\R^{[\a ]+2}$, where 
$[\a]$ is the integer part of $\a$. The target dimension is sharp.  This generalizes a recent result of Wu \cite{Wu}.  
\end{abstract}

\maketitle
\date{\today}

\section{Introduction}

The classical Grushin plane $\mathbb{G}$ is defined as the space $\R^2$ equipped with the sub-Riemannian (Carnot-Carath\'eodory) metric $d_\mathbb{G}$ generated by the vector fields
\[X_1 = \partial_{x_1}  \quad \text{ and } \quad X_2 = x_1\partial_{x_2}.\]
This means more precisely that the distance between points $p,q\in\mathbb{G}$ is
\[ d_{\mathbb{G}}(p,q) = \inf_{\g}  \int_0^1  \sqrt{x_1'(t)^2 + \frac{x_2'(t)^2}{x_1(t)^2}} dt, \]
where the infimum is taken over all paths $\gamma = (x_1(t),x_2(t))\colon [0,1] \to \mathbb{G}$, with $\g(0) = p$ and $\g(1) = q$, that are absolutely continuous in the Euclidean metric. The Grushin plane is one of the simplest examples of a sub-Riemannian manifold, as well as a basic example of the \textit{almost Riemannian} manifolds studied by Agrachev, Boscain, Charlot, Ghezzi, and Sigalotti \cite{ABCGS}, \cite{ABS}. For additional background on the Grushin plane and sub-Riemannian spaces in general, see Bella\"iche \cite{Bellaiche}.  

Recently, Seo \cite{Seo} proved a general characterization of spaces admitting a bi-Lipschitz embedding into some Euclidean space $\R^n$, from which it follows that $\mathbb{G}$ admits such an embedding. In contrast, the Heisenberg group can not be embedded bi-Lipschitz in any Euclidean space \cite{Semmes}. While Seo's result does not give the optimal target dimension, Wu \cite{Wu} constructed an explicit bi-Lipschitz embedding of $\mathbb{G}$ into $\R^3$, where the dimension 3 is the smallest possible. 

In the present article, Wu's result is extended to the generalized Grushin plane $\mathbb{G}_\a$, $\a\geq 0$, studied first by Franchi and Lanconelli \cite{FrLa}. Similarly to $d_\mathbb{G}$, the metric $d_{\mathbb{G}_{\a}}$ is generated by the vector fields
\[X_1 = \partial_{x_1}  \quad \text{ and } \quad X_2 = |x_1|^{\a}\partial_{x_2}.\]
For integer values of $\a$, $|x_1|^{\a}$ can be replaced by $x_1^{\a}$ and the space $\mathbb{G}_\a$ is a sub-Riemannian manifold of step $\a+1$. For noninteger values of $\a$, this space is technically not sub-Riemannian, but this distinction does not matter for the purposes of this paper. Meyerson \cite{Meyerson} and Ackermann \cite{Ack} have shown that $\mathbb{G}_\a $ is quasisymmetric to the Euclidean space $\R^2$ for any $\a \geq 0$. Moreover, it can be deduced by Seo's theorem \cite[Theorem 4.4]{Seo} that $\mathbb{G_\a}$ is bi-Lipschitz embeddable into some Euclidean space when $\a>0$, though without identifying the smallest target dimension. 

In this paper, we construct for each $\a\geq 0$ a bi-Lipschitz embedding of $\mathbb{G}_\a $ into $\R^{[\a ]+2}$ where $[\a]$ is the greatest integer that is less or equal to $\a$. A point of interest in both Wu's and our construction is that the image of $\mathbb{G}_\a$ is a quasiplane in $\R^{[\a]+2}$. 

\begin{thm}\label{thm:main}
For all integers $N\geq 0$ and $n \geq 1$, there exists $L>1$ depending only on $N,n$ such that for any $\a\in [N,N+\frac{n-1}n]$, there exists an $L$-bi-Lipschitz homeomorphism of $\mathbb{G}_{\a}$ onto a $2$-dimensional quasiplane $\mathcal{P}_{\a}$ in $\R^{N+2}$.
\end{thm}

A \emph{$k$-dimensional quasiplane} $\P$ in $\R^n$, with $k<n$, is the image of a $k$-dimensional hyperplane in $\R^n$ under a quasiconformal self-map of $\R^n$. Complete characterizations of these spaces in terms of their geometric structure exist only for $n=2$, $k=1$ by Ahlfors \cite{Ah}. While such intrinsic characterizations have been elusive for $n\geq 3$, several intriguing examples of quasiplanes and quasispheres have been constructed \cite{Bishop, DToro, Lewis, Meyer, PW, VW2, Wu}. 


A couple of remarks are in order. The target dimension $N+2 = [\a ]+2$ in Theorem \ref{thm:main} is minimal. Indeed, by (\ref{eq:grushinQuasimetric}), the \emph{singular line} $\{x_1=0\}$ of $\mathbb{G}_{\a}$ is bi-Lipschitz homeomorphic to the ``snowflaked'' space $(\R,|\cdot| ^{1/(1+\a)})$ which, by a well-known theorem of Assouad \cite[Proposition 4.12]{Assouad}, embeds bi-Lipschitz into $\R^{[\a]+2}$ with the target dimension $[\a]+2$ being the smallest possible when $\alpha>0$. It is noteworthy that, for $\alpha>0$, $\mathbb{G}_\a$ embeds in the same Euclidean space that its singular line embeds in.


The same result of Assouad also justifies the dependence of the constant $L$ on $n$. For if there was a uniform $L$ such that $\mathbb{G}_\a$ was $L$-bi-Lipschitz embeddable in $\R^{N+2}$ for all $\a \in [N,N+1)$, then by a simple Arzel\`a-Ascoli limiting argument (see Lemma \ref{lem:BLembed}), it would follow that $\mathbb{G}_{N+1}$, thus the singular line in $\mathbb{G}_{N+1}$, is also embeddable in $\R^{N+2}$ which is false. 

The following corollary is an immediate consequence of Theorem \ref{thm:main}.

\begin{cor}\label{cor:BLcor}
If $\a\in [0,1)$ then $\mathbb{G}_{\a}$ is bi-Lipschitz homeomorphic to $\mathbb{R}^2$.
\end{cor}

Therefore, $\mathbb{G}_{\a}$ is bi-Lipschitz homeomorphic to $\mathbb{G}_{\beta}$ whenever $\a,\beta \in [0,1)$. In contrast, if $\a \geq 1$ then $\mathbb{G}_\a$ has Hausdorff dimension $\a+1$, and is bi-Lipschitz homeomorphic to $\mathbb{G}_\beta$ only when $\a = \beta$. Combined with the Beurling-Ahlfors quasiconformal extension \cite{BerAhl}, Corollary \ref{cor:BLcor} yields the following result.

\begin{cor}\label{cor:ext}
If $\a\in[0,1)$, then any bi-Lipschitz embedding of the singular line of $\mathbb{G}_{\a}$ into $\R^2$ extends to a bi-Lipschitz homeomorphism of $\mathbb{G}_{\a}$ onto $\R^2$.
\end{cor}

An alternative proof of Corollary \ref{cor:BLcor} along with new results on questions of quasisymmetric parametrizability and bi-Lipschitz embeddability of high-dimensional Grushin spaces can be found in a recent paper of Wu \cite{Wu2}.

\subsection{Outline of the proof of Theorem \ref{thm:main}.}\label{sec:outline}
The proof of Theorem \ref{thm:main} comprises two parts. In Section \ref{sec:rationalproof} we show Theorem \ref{thm:main} for rationals $\a\geq 0$ and in Section \ref{sec:irrationalproof} we use an Arzel\`a-Ascoli limiting argument to prove Theorem \ref{thm:main} for all real values $\alpha\geq 0$. The proof of Corollary \ref{cor:ext} is also given in Section \ref{sec:proofofmainthm}.

Much of the proof of Theorem \ref{thm:main} for rational $\a\geq 0$ follows the method of Wu in \cite{Wu}. The crux of the proof is the construction, for each rational $\a\in[N,N+\frac{n-1}n]$, of a quasisymmetric mapping $F_{\a}:\R^{N+2}\to \R^{N+2}$, such that in each ball $B(x,\frac{1}{2}\dist(x,\{0\}\times\R))$, $F_{\a}$ is the product of $\dist(x,\{0\}\times\R)^{-\frac{1}{1+\a}}$ and a $\lambda$-bi-Lipschitz mapping with $\lambda$ depending only on $N,n$. Such a mapping is $\frac{1}{1+\a}$-\emph{snowflaking} on $\{0\}\times\R \subset \R^{N+1}\times\R$ (i.e. $|F_{\a}(x) - F_{\a}(y)| \simeq |x-y|^{\frac{1}{1+\a}}$ for all $x,y \in \{0\}\times\R$) and maps the $2$-dimensional plane $\R\times\{0\}\times\R$ onto a quasiplane $\mathcal{P}_{\a}$. 
Composed with a quasisymmetric homeomorphism of $\mathbb{G}_\alpha$ onto $\R^2$, we obtain a bi-Lipschitz homeomorphism $f_{\a}$ of $\mathbb{G}_\alpha$ into $\mathcal{P}_{\a}$.

The quasisymmetric mappings $F_\a$ are constructed in Section \ref{sec:proprational} by iterating a finite number of bi-Lipschitz mappings $\Theta$ which are defined in Section \ref{sec:blocks} as in \cite{Wu}. However, a straightforward generalization of Wu's method, without additional care, would give no control on the local bi-Lipschitz constant $\lambda$ (thus on the bi-Lipschitz constant of $f_{\a}$), and the proof of Theorem \ref{thm:main} for irrational values of $\a$ would not be possible. To overcome this issue, we construct in Section \ref{sec:blocks} two sets of bi-Lipschitz mappings $\Theta_z$, corresponding to $z = 0$ and $z = N+\frac{n-1}n$, and then periodically alternate between these when constructing the quasisymmetric mapping $F_\alpha$.

Our inability to define $F_{\a}$ for irrational $\a>0$ is the reason for considering the irrational case separately; see Remark \ref{rem:distinction}. 


\medskip 

\textbf{Acknowledgements.} The authors are grateful to Jang-Mei Wu and Jeremy Tyson for suggesting this problem and for valuable discussions.    

\section{Preliminaries}

A homeomorphism $f\colon D\to D'$ between two domains in $ \mathbb{R}^n$ is called $K$-\emph{quasiconformal}  if it is orientation-preserving, belongs to $ W_{\text{loc}}^{1,n}(D)$, and satisfies the distortion inequality
\[|Df(x)|^n \le K J_f(x) \quad \text{a. e.} \,\,\, x \in D,\]
where $Df$ is the formal differential matrix and $J_f$ is the Jacobian. 

An embedding $f$ of a metric space $(X,d_X)$ into a metric space $(Y,d_Y)$ is said to be $\eta$-\emph{quasisymmetric} if there exists a homeomorphism $\eta \colon [0,\infty) \to [0,\infty)$ such that for all $x,a,b \in X$ and $t>0$ with $d_X(x,a) \leq t d_X(x,b)$,
\[d_Y(f(x),f(a)) \leq \eta(t)d_Y(f(x),f(b)). \]

A quasisymmetric mapping between two domains in $\R^n$ is quasiconformal. On the other hand, a quasiconformal mapping defined on a domain $D\subset \R^n$ is quasisymmetric on each compact set $E \subset D$. In $\R^n$ the two notions coincide: if $f:\R^n \to \R^n$ is $K$-quasiconformal then it is $\eta$-quasisymmetric for some $\eta$ depending only on $K,n$. For a systematic treatment of quasiconformal mappings see \cite{Vais1}.

A mapping $f\colon X \to Y$ between metric spaces is \emph{$L$-bi-Lipschitz} if there exists a constant $L \geq 1$ such that $L^{-1}d_X(x,y) \leq d_Y(f(x),f(y)) \leq Ld_X(x,y)$ for all $x,y \in X$.

In the following, we write $u\lesssim v$ (resp. $u \simeq v$) when the ratio $u/v$ is bounded above  (resp. bounded above and below) by positive constants. These constants may vary, but are described in each occurrence.

\section{Basic geometric constructions}\label{sec:blocks}
This section extends the construction by Wu \cite{Wu} to higher-dimensional targets; the notational conventions follow those of Wu as much as possible. Our goal is to build certain annular tubes and bi-Lipschitz maps between these tubes which are used in Section \ref{sec:proprational} to define quasiconformal homeomorphisms of $\R^{N+2}$. 
These constructions are based on examples of Bonk and Heinonen \cite{BH} and Assouad \cite{Assouad}.

\subsection{Definitions and notation}
An \emph{$N$-cube} $\mathcal{C}$ is the product $\D_1\times\cdots\times\D_N$ of bounded closed intervals $\D_i\subset\R$ of equal length. A \emph{$j$-face} of $\mathcal{C}$ is a product $\D_1'\times\cdots\times\D_N'$ where, for $j$ indices, $\D_i' = \D_i$ and for the other $N-j$ indices $\D_i'$ is an endpoint of $\D_i$. The $0$-faces of a cube $\mathcal{C}$ are its \emph{vertices}.

For an $N$-cube $\mathcal{C}$ and integer $0 \leq k \leq N$, we define a \emph{$k$-flag} of $\mathcal{C}$ to be a sequence $\{\mathcal{C}^j\}_{j=0}^k$ where $\mathcal{C}^j$ is a $j$-face of $\mathcal{C}$ and $\mathcal{C}^{j-1} \subset \mathcal{C}^j$ for all $1 \leq j \leq k$. Observe that for $N$-cubes $\mathcal{C}$ and $\widetilde{\mathcal{C}}$ and $(N-2)$-flags $\{\mathcal{C}^j\}$ and $\{\widetilde{\mathcal{C}}^j\}$, there exists a unique orientation-preserving similarity $\psi: \R^N \to \R^N$ such that $\psi(\mathcal{C})= \widetilde{\mathcal{C}}$ and $\psi(\mathcal{C}^j) = \widetilde{\mathcal{C}}^j$ for each $0 \leq j \leq N-2$.

For a point $x = (x_1,\dots,x_{N}) \in \R^{N}$ and a number $r>0$, define the cube 
\[ \mathcal{C}^{N}(x,r) = [x_1-r/2, x_1 +r/2]\times \dots \times [x_{N}-r/2,x_{N}+r/2]  \] 
and denote $\mathfrak{C}^{N} = \mathcal{C}^{N}(0,1)$ where $0$ here denotes the origin in $\R^{N}$. 

Slightly abusing the notation, we define for two numbers $0<r<R<\infty$ the \emph{cubic annulus}
\[ \mathcal{A}^{N}(r,R) = \overline{(R\mathfrak{C}^{N}) \setminus (r\mathfrak{C}^{N})} = [-R/2,R/2]^N \setminus (-r/2,r/2)^N.\]
Here and for the rest, for $X\subset \R^N$ and $c > 0$, we write $c X = \{c x : x\in X\}$.

Finally, for a polygonal arc $\ell \subset \R^{N}$ and some $\e>0$, define the \emph{cubic thickening} of $\ell$ 
\[ \mathcal{T}^{N}(\ell,\e) = \overline{\bigcup \mathcal{C}^{N}(x,\e)}\]
where the union is taken over all $x\in \ell$ such that their distances from the endpoints of $\ell$ are at least $\e/2$.

For the rest of Section \ref{sec:blocks} we fix integers $N\geq 0$, $n\geq 1$ and set 
\[ p = p_{N,n} = N+\frac{n-1}{n}\text{ and }M = M_{N,n} = 9^{n(N+2)}.\] 
The dependence of quantities and sets on $N,n$ is omitted whenever possible.

\subsection{Blocks}\label{sec:tubes} 

Let $I\subset\mathfrak{C}^{N+1}\times[0,1]$ be the straight-line path from $(0,\dots,0)$ to $(0,\dots,0,1)$ and $L\subset\mathfrak{C}^{N+1}\times[0,1]$ be the straight-line path from $(0,\dots,0)$ to $(0,\dots,0,\frac{1}2)$ concatenated with the straight-line path from $(0,\dots,0,\frac{1}2)$ to $(\frac{1}2,0,\dots,0,\frac{1}2)$.
    
We define three types of blocks that are used throughout the paper: 
\begin{enumerate}
\item the \emph{$I$-block} $Q_I = \mathcal{T}^{N+2}(I,\frac{M-2}M) = (\frac{M-2}{M}\mathfrak{C}^{N+1})\times [0,1]$;
\item the \emph{$L$-block} $Q_L = \mathcal{T}^{N+2}(L,\frac{M-2}M) $
\[= (\frac{M-2}{M}\mathfrak{C}^{N+1}\times[0,\frac{M-1}{M}])\cup ([\frac{1}{2},1]\times \frac{M-2}{M}\mathfrak{C}^{N+1}) \]
\item the \emph{regular block} $\mathsf{Q} = \mathfrak{C}^{N+1}\times[0,1]$.
\end{enumerate}
On each of these blocks, the entrance, exit and side are defined as follows. 
\begin{enumerate}
\item The \emph{entrance} of $Q_{I}$ is $\text{en}(Q_{I}) = Q_{I} \cap \{x_{N+2} = 0\}$,
\item the \emph{exit} of $Q_I$ is $\text{ex}(Q_I) = Q_I \cap \{x_{N+2} = 1\}$,
\item the \emph{side} of $Q_{I}$ is $\text{s}(Q_{I}) = \overline{\partial Q_{I} \setminus(\text{en}(Q_{I}) \cup \text{ex}(Q_{I}))}$.
\end{enumerate}
Analogous definitions can be made for $Q_L$ and $\mathsf{Q}$ with the difference that the \emph{exit} of $Q_L$ is $\text{ex}(Q_L) = Q_L \cap \{x_{1} = \frac12\}$. These definitions are applied to images of the respective objects under similarity maps. For a similarity map $h$ and $\ell \in \{h(I), h(L)\}$, we write $Q_\ell$ in place of $h(Q_I)$ or $h(Q_L)$. We call $Q_\ell$ the \emph{block associated with the segment $\ell$}; note that $Q_\ell$ naturally inherits a direction from $\ell$. 

\subsection{Cores}\label{sec:cores}
From each block $Q_I$, $Q_L$ and $\mathsf{Q}$ we remove a \emph{core} from its interior, which we describe in this section.

In Section \ref{sec:paths} we construct a simple polygonal path $J_I = J_I(N,n) \subset Q_I$ from $(0,\ldots, 0, 0)$ to $(0, \ldots, 0, 1)$ consisting of $M^{1+p}$ many $I$- and $L$-segments $\ell_1, \ldots, \ell_{M^{1+p}}$ of length $1/M$ labelled according to their order in $J_I$ with the following properties.
\begin{enumerate}
\item The segments $\ell_1$, $\ell_{M^{1+p}}$, and $\ell_{(M^{1+p}+1)/2}$ are $I$-segments.
\item For all $1 \leq m < M^{1+p}$, $Q_{\ell_m} \cap Q_{\ell_{m+1}}$ is the exit of $Q_{\ell_m}$ and the entrance of $Q_{\ell_{m+1}}$. If $1 \leq l, m \leq M^{1+p}$ and $|m-l|>1$, then $Q_{\ell_m} \cap Q_{\ell_l} = \emptyset$.  
\item $\text{en}(Q_{\ell_1}) = Q_{\ell_1} \cap \partial Q_I \subset \text{en}(Q_{I})$ and $\text{ex}(Q_{\ell_{M^{1+p}}}) = Q_{\ell_{M^{1+p}}} \cap \partial Q_I \subset \text{ex}(Q_{I})$. For $2 \leq m \leq M^{1+p}-1$, $Q_{\ell_m} \cap \partial Q_I = \emptyset$.   
\item $J_I$ is symmetric with respect to the plane $x_{N+2} = \frac12$.
\item $J_I$ is unknotted in $Q_I$, in the sense that there every bi-Lipschitz homeomorphism $\theta: (\partial Q_I,J_I) \to (\partial\mathsf{Q},I)$ extends to a bi-Lipschitz homeomorphism $\Theta : Q_I \to \mathsf{Q}$. 
\end{enumerate}
Similarly, in Section \ref{sec:paths} we construct a simple polygonal path $J_L = J_L(N,n) \subset Q_L$ satisfying the same properties, except that $\ell_{(M^{1+p}+1)/2}$ is an $L$-segment and $J_L$ is symmetric with respect to the plane $x_1 + x_{N+2} = \frac12$. 

Given $J_I = \bigcup_{m=1}^{M^{1+p}} \ell_m$ as above, define the \emph{core} 
\[\k_p(Q_I) = \bigcup_{m=1}^{M^{1+p}} Q_{\ell_m} = \mathcal{T}^{N+2}(J_I,\frac{M-2}{M^2}).\]
We similarly define the core $\k_p(Q_L)$. The entrance, the exit and the side of $\k_p(Q_I),\k_p(Q_L)$ are canonically defined. A second set of cores $\k_0(Q_I),\k_0(Q_L)$ in $Q_I,Q_L$, respectively, is defined as follows. Write $I = \bigcup_{m=1}^{M}\ell_m$ with $\ell_m = \{0\}\times[m-1,m] \subset \R^{N+1}\times I$ and set
\[ \k_{0}(Q_I) = \bigcup_{m=1}^M Q_{\ell_m} = \mathcal{T}^{N+2}(I,\frac{M-2}{M^2}).\] 
Similarly write $L = \bigcup_{m=1}^{M}\ell_m'$ where $\ell_{m}'$ is an $L$-segment if $m=\frac{M+1}{2}$ and an $I$-segment otherwise; and each $\ell_m'$ has length $1/M$. Set $\k_{0}(Q_L) = \bigcup_{m=1}^M Q_{\ell'_m} = \mathcal{T}^{N+2}(L,\frac{M-2}{M^2})$.

To simplify the notation, in what follows we write $Q_m$ instead of $Q_{\ell_m}$. 

Two types of cores are similarly defined for the regular block $\mathsf{Q}$. For each $z \in \{0,p\}$ let
\[ \mathsf{k}_z(\mathsf{Q}) = (M^{-1-z}\mathfrak{C}^{N+1}) \times [0,1]\]
which is composed of $M^{1+z}$ consecutive blocks 
\[\mathsf{Q}_m = M^{-1-z}\left ( \mathfrak{C}^{N+1} \times [m-1, m]\right ), \quad m = 1, \ldots, M^{1+z}.\]

\subsection{Flag-edges}\label{sec:edges}
We introduce in this section \emph{flag-edges} and \emph{flag-paths}, which generalize the \emph{edges} and \emph{edge paths} used by Wu \cite[Section 2.3]{Wu} to blocks of arbitrary dimensions. These play an important bookkeeping role later when defining bi-Lipschitz maps between annular tubes.

For the rest fix an $(N-1)$-flag $\mathcal{F}_0 = \{\mathcal{C}^j\}_{j=1}^{N-1}$ of $\mathfrak{C}^{N+1}$. We call the collection of faces $e_{\mathcal{F}_0} = \{(\frac{M-2}{M}\mathcal{C}^j)\times [0,1]\}_{j=0}^{N-1}$ a \emph{flag-edge} on $Q_I$.

Before defining flag-edges on $Q_L$, we first define faces on the side $\text{s}(Q_L)$ inductively. If $\mathcal{C}^0 = (x_1,\dots,x_{N+1})$, $x_j \in \{\pm \frac{M-2}{2M}\}$, is a $0$-face of $\text{en}(Q_L)$ then define the $L$-type path
\[ P(\mathcal{C}^0) = (\{(x_1,\dots,x_{N+1})\}\times [0,\frac12-x_{1}]) \cup ([x_1,\frac12]\times\{(x_{2},\dots,x_{N+1},\frac12-x_{1})\}.\]
Suppose that for every $j$-face $\mathcal{C}^j$ of $\text{en}(Q_L)$, the set $P(\mathcal{C}^j)$ has been defined. Let $\mathcal{C}^{j+1}$ be a $(j+1)$-face of $\text{en}(Q_L)$ and let $\mathcal{C}^{j}_1, \dots, \mathcal{C}^{j}_{2(j+1)}$ be the $j$-faces of $\mathcal{C}^{j+1}$. Then define $P(\mathcal{C}^{j+1})$ to be the union of all line segments with endpoints on $\bigcup_{i=1}^{2(j+1)} P(\mathcal{C}^{j}_{i})$ that lie entirely on $\text{s}(Q_L)$. We call $P(C^{j+1})$ a $(j+2)$-face on $\partial Q_L$. 

Let now $\mathcal{F} = \{\mathcal{C}^j\}_{j=0}^{N-1}$ be an $(N-1)$-flag of $\mathfrak{C}^{N+1}$. We call the collection $e_{\mathcal{F}} = \{P(\frac{M-2}{M}\mathcal{C}^j)\}_{j=0}^{N-1}$ a flag-edge on $Q_L$.

We now define \emph{flag-paths} along the cores $\k_z(Q_I)$, $\k_z(Q_L)$ for each value $z \in \{0, p\}$. We start with the $Q_L$ case. Rescaling an $(N-1)$-flag $\mathcal{F}$ of $\mathfrak{C}^{N+1}$, we obtain an $(N-1)$-flag $\mathcal{F}_1$ on the entrance of the first block $Q_1$ of $\k_z(Q_L)$. For a $j$-face $\mathcal{C}^j_1 \in \mathcal{F}_1$ define $P(\mathcal{C}^j_1) \subset \text{s}(Q_1)$ as above and note that $P(\mathcal{C}^j_1)$ defines uniquely a $j$-face $\mathcal{C}^j_2$ on the entrance of the block $Q_2$. Continuing inductively we obtain $j$-faces $\mathcal{C}^j_m$ on $\text{en}(Q_m)$ and $(j+1)$-faces $P(\mathcal{C}^j_m)$ on $\text{s}(Q_m)$. Define the \emph{flag-path} $w_{\mathcal{F}} = \{ \bigcup_{m=1}^{M^{1+z}}P(\mathcal{C}^j_m)\}_{j=0}^{N-1}$. For each block $Q_m$ in $\k_z(Q_L)$, $m \in \{1, \ldots, M^{1+z}\}$, we call $w_{\mathcal{F}} \cap Q_m$ the \emph{marked flag-edge of $Q_m$ derived from the data $(Q_L, e_{\mathcal{F}})$}.

A corresponding flag-path $w_{\mathcal{F}_0}$ is defined similarly for $\kappa_z(Q_I)$. For this we use the flag $\mathcal{F}_0$ instead of an arbitrary $(N-1)$-flag $\mathcal{F}$ of $\mathfrak{C}^{N+1}$. 

In addition, let $\mathsf{e} = \{\mathcal{C}^j \times [0,1] : \mathcal{C}_j \in \mathcal{F}_0\}$ be a flag-edge of $\mathsf{Q}$ and $\mathsf{w} = \{(M^{-1-z}\mathcal{C}^j) \times [0,1] : \mathcal{C}_j \in \mathcal{F}_0\}$ be a flag-path along $\mathsf{k}_z(\mathsf{Q})$. As before we omit the dependency on $z$ in the notation for $\mathsf{w}$.

\subsection{Annular tubes}\label{sec:annulartubes}

For each $z \in \{0,p\}$, define the \emph{annular tubes}
\[ \tau_{z}(Q_I) = \overline{Q_I \setminus \k_{z}(Q_I)} \text{, } \tau_{z}(Q_L) = \overline{Q_{L} \setminus \k_{z}(Q_L)} \text{ and }\mathsf{t}_z(\mathsf{Q}) = \overline{\mathsf{Q}\setminus\mathsf{k}_z(\mathsf{Q})}.\]
For $Q \in \{Q_I,Q_L\}$, we define the \emph{entrance} and \emph{exit} of each $\tau_z(Q)$ as $\text{en}(Q) \cap \tau_z(Q)$ and $\text{ex}(Q) \cap \tau_z(Q)$, respectively. These are isometric to the cubic annulus $A = \frac{M-2}{M}\mathcal{A}^{N+1}(\frac{1}{M},1)$. The remaining part of $\partial \tau_{z}(Q_I)$ is composed of the side $\text{s}(Q_I)$ of block $Q_I$ and the side $\text{s}(\k_{z}(Q_I))$ of the core $\k_{z}(Q_I)$. The boundary of $\tau_{z}(Q_L)$ can be similarly partitioned.

Define similarly the entrance and exit of $\mathsf{t}_z(\mathsf{Q})$. These are isometric to the cubic annulus $\mathsf{A}_z = \mathcal{A}^{N+1}(M^{-1-z},1)$. Note that $\mathsf{A}_z$ depends on $z$ while $A$ does not.\

If $\sigma$ is a similarity mapping of $Q_{I}$ onto some block $\sigma(Q_I)$, we denote by $\k_{z}(\sigma(Q_I))$ the image $\sigma(\k_{z}(Q_I))$ with $z\in\{0,p\}$. The sets $\k_{z}(\sigma(Q_L))$,  $\mathsf{k}_z(\sigma(\mathsf{Q}))$, $\tau_z(\sigma(Q_I))$, $\tau_z(\sigma(Q_L))$ and $\mathsf{t}_z(\sigma(\mathsf{Q}))$ are defined similarly when $\sigma$ is a similarity mapping.

\subsection{Bi-Lipschitz maps between annular tubes}\label{sec:BLmaps}
For each $z \in \{0, p\}$, each $Q\in \{Q_I, Q_L\}$, and every $(N-1)$-flag $\mathcal{F}$ of $\mathfrak{C}^{N+1}$, we define in this section bi-Lipschitz homeomorphisms $\Theta_z^{\mathcal{F}}: (\mathsf{t}_z(\mathsf{Q}), \mathsf{e}, \mathsf{w}) \to (\tau_z(Q), e_{\mathcal{F}}, w_{\mathcal{F}})$ where $\mathcal{F}=\mathcal{F}_0$ if $Q=Q_I$.

The construction of these maps is performed in 4 steps. In Step 1 we define the mappings on $\text{s}(\mathsf{Q})$, in Step 2 we define them on $\text{s}(\mathsf{k}_z(\mathsf{Q}))$ and in Step 3 we define them on the entrance and exit of $\mathsf{t}_z(\mathsf{Q})$. Combining the first three steps we obtain bi-Lipschitz mappings $\theta_z^{\mathcal{F}}: (\partial \mathsf{t}_z, \mathsf{e}, \mathsf{w}) \to (\partial \tau_z(Q), e_{\mathcal{F}}, w_{\mathcal{F}})$. Finally, in Lemma \ref{lem:BLext} we extend the mappings on the whole $\mathsf{t}_z(\mathsf{Q})$. 

For each $(N-1)$-flag $\mathcal{F}$ of $\mathfrak{C}^{N+1}$ let $\psi_{\mathcal{F}}$ be the unique rotation on $\R^{N+2}$ that maps $\mathfrak{C}^{N+1}$ onto itself and $\mathcal{F}$ onto $\mathcal{F}_0$.

\emph{Step 1}. Define $\theta_z^{\mathcal{F}_0}: (\text{s}(\mathsf{Q}), \mathsf{e}) \to (\text{s}(Q_I), e_{\mathcal{F}_0})$ by $\theta_z^{\mathcal{F}_0}(x,t) = (\frac{M-2}{M}x,t)$, where $x \in \partial \mathfrak{C}^{N+1}$ and $t \in [0,1]$. To define $\theta_z^{\mathcal{F}}$ onto $(\text{s}(Q_L), e_{\mathcal{F}})$ first observe that $\text{s}(Q_L)$ is the union of $L$-type $1$-fibers
\begin{align*}
L_x = &(\{\frac{M-2}{M}(x_1,\dots,x_{N+1})\}\times [0,1/2-\frac{M-2}{M}x_1]) \\
&\cup ([\frac{M-2}{M}x_1,1/2]\times\{(0,\dots,0,1/2)+\frac{M-2}{M}(x_2,\dots,x_{N+1},-x_1)\}
\end{align*}
where $x =(x_1,\dots,x_{N+1})\in\partial \mathfrak{C}^{N+1}$. Similarly, $\text{s}(\mathsf{Q})$ is the union of $1$-fibers $I_{x} = \{x\}\times [0,1]$ where $x\in \partial \mathfrak{C}^{N+1}$. Define $\theta_z^{\mathcal{F}}$ on $\text{s}(\mathsf{Q})$ by mapping each $I_{x}$ to $L_{\psi_{\mathcal{F}}(x)}$ by arc-length parametrization. Note that for this step, the mappings $\theta_z^{\mathcal{F}}$ do not actually depend on $z$.

\emph{Step 2}. We extend each $\theta_z^{\mathcal{F}}$ to the inner side $\text{s}(\mathsf{k}_z(\mathsf{Q}))$ of $\partial\mathsf{t}_z(\mathsf{Q})$. Given a block $\mathsf{Q}_m$ of $\mathsf{k}_z(\mathsf{Q})$ let $\zeta_z^m$ be the similarity map in $\mathbb{R}^{N+2}$ that maps $(\mathsf{Q}, \mathsf{e})$ onto $(\mathsf{Q}_m, \mathsf{w} \cap \mathsf{Q}_m)$. Similarly, given a block $Q_m$ in the core $\k_z(Q)$, let $\varepsilon(Q_m)$ be the marked flag-edge $w_{\mathcal{F}} \cap Q_m$ derived from $(Q, e_{\mathcal{F}})$ and let $\sigma_z^m: (Q(Q_m), e_{\mathcal{F}(Q_m)}) \to (Q_m, \varepsilon(Q_m))$ be a similarity map for some unique $Q(Q_m) \in \{Q_I, Q_J\}$ and $(N-1)$ flag $\mathcal{F}(Q_m)$ of $\mathfrak{C}^{N+1}$. (The dependence of $\sigma_z^m$ on $\mathcal{F}$ is omitted to simplify the notation.)  Define now $\theta_z^{\mathcal{F}}$ on the inner side $\text{s}(\mathsf{k}_z(\mathsf{Q}))$ by taking $\theta_z^{\mathcal{F}}|\text{s}(\mathsf{Q}_m) = \sigma_z^m \circ \theta_z^{\mathcal{F}(Q_m)} \circ (\zeta_z^m)^{-1}$ for all $1 \leq m \leq M^{1+z}$. Since the union of marked flag-edges of the $Q_m$ is the flag-path $w_{\mathcal{F}}$, the map $\theta_z^{\mathcal{F}}$ is defined consistently on the intersection of consecutive blocks and thus is well-defined.

\emph{Step 3}. For each $(N-1)$-flag $\mathcal{F}$ of $\mathfrak{C}^{N+1}$, let $\phi_z^{\mathcal{F}} : \mathsf{A}_z \to A$ with
\[ \phi_z^{\mathcal{F}}(xt) = \frac{M-2}{M}\left ( \frac{M-1}{M-M^{-z}}(t-1)+1 \right ) \psi_{\mathcal{F}}(x)\]
where $t\in[M^{-1-z},1]$ and $x\in \partial\mathfrak{C}^{N+1}$. Define $\theta_z^{\mathcal{F}}$ on the entrance and exit of $\mathsf{t}_z(\mathsf{Q})$ by $\phi_z^{\mathcal{F}}$ modulo an isometry chosen in such a way that the mappings $\theta_z^{\mathcal{F}}: (\partial \mathsf{t}_z(\mathsf{Q}), \mathsf{e}, \mathsf{w}) \to (\partial \tau_z(Q), e_{\mathcal{F}}, w_{\mathcal{F}})$ are homeomorphisms. Then $\theta_z^{\mathcal{F}}$ are in fact bi-Lipschitz.

The final step in the construction of the mappings $\Theta_z^{\mathcal{F}}$ is given in the next lemma.

\begin{lem}\label{lem:BLext}
Every bi-Lipschitz map $\theta_z^{\mathcal{F}} $ extends to a bi-Lipschitz map
\[ \Theta_z^{\mathcal{F}} : (\mathsf{t}_z(\mathsf{Q}), \mathsf{e}, \mathsf{w}) \to (\tau_z(Q), e_{\mathcal{F}}, w_{\mathcal{F}}).\]
\end{lem}

We verify the lemma first for $z=0$. If $Q=Q_I$ then define $\Theta_0^{\mathcal{F}_0} : \mathsf{t}_0(\mathsf{Q}) \to \tau_0(Q_I)$ with $\Theta_0^{\mathcal{F}_0}(xt,t') = (\phi_{0,\mathcal{F}_0}(xt),t')$. If $Q=Q_L$ then note that $\tau_0(Q_I)$ is the union of $1$-fibers $I_x$ and $\tau_0(Q_L)$ is the union of $1$-fibers $L_x$ where $I_x,L_x$ are as in Step 1 and $x\in \frac{M-2}{M}\mathcal{A}^{N+1}(M^{-1},1)$. Let $\theta_{\mathcal{F}} : \tau_0(Q_I) \to \tau_0(Q_L)$ be the bi-Lipschitz mapping that maps each $I_{x}$ to $L_{\psi_{\mathcal{F}}^{-1}(x)}$ by arc-length parametrization. Set $\Theta_0^{\mathcal{F}} = \theta_{\mathcal{F}}\circ\Theta_0^{\mathcal{F}_0}$.

The proof of Lemma \ref{lem:BLext} when $z=p$ relies on the structure of the paths $J_I$, $J_L$ and is deferred until Section \ref{sec:exten_proof}.

\section{Quasisymmetric snowflaking homeomorphisms in $\R^{N+2}$}\label{sec:proprational}

The key part of the proof of Theorem \ref{thm:main} for a rational $\a\in[N,N+\frac{n-1}{n}]$ is the construction of a quasisymmetric mapping $F_{\a} \colon \R^{N+2} \to \R^{N+2}$ that maps Whitney squares of a $2$-dimensional plane of $\R^{N+2}$ into sets which are bi-Lipschitz homeomorphic to the Whitney squares of $\mathbb{G}_{\a}$. 



\begin{prop}\label{prop:qsimilarity}
For all integers $N\geq 0$ and $n\geq 1$, there exists $\lambda >1$ depending only on $N,n$ satisfying the following. For each rational $\a\in[N,N+\frac{n-1}{n}]$, there exists an $\eta$-quasisymmetric map $F_{\a} : \R^{N+2} \to \R^{N+2}$ with $\eta$ depending only on $N,n$ such that, 
\begin{equation}\label{eq:quasisim2} 
|x'|^{\frac{\a}{1+\a}}F_{\a}|_{B(x,\frac{1}{2}|x'|)} \quad \text{ is }\lambda\text{-bi-Lipschitz}
\end{equation}
for all $x = (x',x'')\in\R^{N+1}\times\R$ with $|x'| \neq 0$.
\end{prop}

The mapping $F_{\a}$ is $\frac{1}{1+\a}$-snowflaking on $\{0\}\times\R \subset \R^{N+1}\times\R$.

\begin{cor}\label{cor:qsimilarity}
Let $F_{\a} \colon \R^{N+2} \to \R^{N+2}$ be the mapping of Proposition \ref{prop:qsimilarity}. Then, there exists $\lambda'>1$ depending only on $N,n$ such that 
\[ (\lambda')^{-1}|x-y|^{\frac{1}{1+\a}} \leq |F_{\a}(x) - F_{\a}(y)| \leq \lambda'|x-y|^{\frac{1}{1+\a}}\] 
for all $x = (x',x''),y=(y',y'') \in \R^{N+1}\times\R$ with $|x''-y''| \geq \frac{1}{2}\max\{|x'|,|y'|\}$.
\end{cor}

\begin{proof}
Let $\hat{x} = (w,x'')$ and $\hat{y} = (w,y'')$ where $w\in\R^{N+1}$ satisfies $|w|=3|x-y|$. Note that $|x-y| \simeq |\hat{x}-x| \simeq |\hat{x}-\hat{y}|$. By (\ref{eq:quasisim2}) we have that $|F_{\a}(\hat{x}) - F_{\a}(\hat{y})| \simeq |w|^{-\frac{\a}{1+\a}}|x-y|$ and applying the fact that $F_{\a}$ is quasisymmetric twice (to the points $x,y,\hat{x}$ and to the points $\hat{x},x,\hat{y}$), $|F_{\a}(x) - F_{\a}(y)| \simeq |F_{\a}(\hat{x}) - F_{\a}(\hat{y})| \simeq |x-y|^{\frac{1}{1+\a}}$.
\end{proof}

The rest of this section is devoted to the proof of Proposition \ref{prop:qsimilarity}. As mentioned in Section \ref{sec:outline}, the construction in \cite{Wu} can be used to deduce Proposition \ref{prop:qsimilarity} for all rational $\a\in [N,N+1)$ but with no control on $\lambda$ and $\eta$. For this reason, while in \cite{Wu} the mapping $F_{\a}$, for $\a=1$, is obtained by iterating one family of bi-Lipschitz mappings $\Theta^{\mathcal{F}}$, here $F_{\a}$ is obtained by a periodic iteration of $2$ families of bi-Lipschitz mappings $\Theta^{\mathcal{F}}_{z}$ (for the two values $z\in\{0,N+\frac{n-1}{n}\}$) in the alternating fashion of Section \ref{sec:prelimarr}.

For the rest we fix integers $N\geq 0$, $n\geq 1$ and a rational number $\a\in[N,p_{N,n}]$ where, as before, $p_{N,n} = N+\frac{n-1}{n}$. In Section \ref{sec:qcinQ} we define the map $F_{\a}$ on the block $\mathsf{Q}$ and in Section \ref{sec:QCext} we extend it in all $\R^{N+2}$.

\subsection{A preliminary arrangement}\label{sec:prelimarr}
Suppose that $\frac{\a}{p_{N,n}} = \frac{a}{a+b}$ for some $a,b\in\N$. Using the next lemma we create a periodic sequence $(z_k)_{k\geq 1}$ that takes only the two values $0,p_{N,n}$ and $|z_1+\dots +z_k- k\a| \leq p_{N,n}$ for all $k \geq 1$.

\begin{lem}\label{lem:sum}
Suppose that $y<z<x$ are such that $(a+b)z = ax +by$ for some nonnegative integers $a,b$. Then, there exists a finite sequence $(z_k)_{1}^{a+b}$ which has $a$ terms $x$ and $b$ terms $y$ such that, for all $k=1,\dots,a+b$,
\begin{equation}\label{eq:sum}
|z_1+\dots +z_k - kz| \leq x-y.
\end{equation}
\end{lem}

\begin{proof}
We may assume that $a,b\neq 0$; otherwise the claim is immediate. 

Define $z_k$ inductively as follows. Set $z_1 = x$. Suppose that the terms $z_1,\dots,z_k$ have been defined; if $z_1+\dots + z_k \geq kz$, set $z_{k+1} = y$ and if $z_1+\dots z_k < kz$, set $z_{k+1} = x$.

Suppose that for some $k_0 < a+b$, the sequence $\{z_1,\dots,z_{k_0}\}$ contains exactly $b$ terms $y$. Then, $z_1+\dots + z_{k_0} - k_0z = (a+b-k_0)(z-x) <0$ and thus, $z_{k_0+1} = x$. Similarly, $z_k = x$ for all $k=k_0+1,\dots,a+b$ and $(z_k)^{a+b}_1$ has exactly $b$ terms $y$ and $a$ terms $x$. The same arguments apply if, for some $k_0 < a+b$, the sequence $\{z_1,\dots,z_{k_0}\}$ contains exactly $a$ terms $x$.

To show (\ref{eq:sum}) we apply induction on $k$. If $k=1$ then $|z_1-z| = |x-z| < x-y$. Suppose that (\ref{eq:sum}) is true for some $k<a+b$. Without loss of generality, assume that $z_1+\dots+z_{k} - k z \geq 0$. Then, $z_{k+1} = y$ and
\[y-x < y-z \leq z_1+\dots+z_{k+1} - (k+1)z \leq (x-y) + (y-z) \leq x-y.\qedhere\]
\end{proof}

By Lemma \ref{lem:sum}, there exists a sequence $(z_k)_{1}^{a+b}$ having $a$ terms $p_{N,n}$ and $b$ terms $0$ such that, for each $k = 1,\dots,a+b$,
\[ |z_1 + \dots + z_k - k\a| \leq p_{N,n}.\]
Extend $z_k$ to all $k\in\N$ with $z_k = z_{k'}$ if $k \equiv k' \mod (a+b)$.

\subsection{A quasiconformal map on $\mathsf{Q}$}\label{sec:qcinQ}

We define a quasiconformal mapping $f: (\mathsf{Q},\mathsf{e}) \to (Q_I, e_{\mathcal{F}_0})$, iterating the mappings $\Theta^{\mathcal{F}}_z$ as in \cite[Section 2.6]{Wu}.

Set $\mathsf{K}_0 = \mathsf{t}(\mathsf{Q})$ and for $k\geq 1$,
\[ \mathsf{K}_{-k} = (M^{-k-z_1-\dots-z_{k}}\mathfrak{C}^{N+1})\times [0,1].\] 
Moreover, define $\mathsf{T}_{-k} = \overline{\mathsf{K}_{-k} \setminus \mathsf{K}_{-k-1}}$ with $k \geq 0$. Then 
\[ \mathsf{Q} = (\{0\}\times[0,1]) \cup \bigcup_{k \geq 0} \mathsf{T}_{-k}.\]
Set also $K_0 = Q_I$, $K_{-1} = \k_{z_1}(Q_I)$, and $T_0 = \overline{K_0 \setminus K_{-1}} = \tau_{z_1}(Q)$. Let $f|\mathsf{T}_0 = \Theta_{z_1}^{\mathcal{F}_0}: \mathsf{T}_0 \to T_0$ noticing that $\mathsf{T}_0 = \mathsf{t}_{z_1}(\mathsf{Q})$.  

For every $l \in [1, M^{1+z_1}]$, let $\varepsilon_l$ be the marked flag-edge on $Q_l$ derived from $(Q_I, e_{\mathcal{F}_0})$, and let $\sigma^l_{z_1}: (Q(l), e_{\mathcal{F}(l)}) \to (Q_l, \varepsilon_l)$ be the similarity mapping defined in Section \ref{sec:BLmaps} where $Q(l)\in\{Q_I,Q_L\}$ and $\mathcal{F}(l)$ is a $(N-1)$-flag of $\mathfrak{C}^{N+1}$. The similarity $\sigma^l_{z_1}$ induces naturally a core $\k_{z_2}(Q_l)$, consequently a tube $\tau_{z_2}(Q_l) = \overline{Q_l \setminus \k_{z_2}(Q_l)}$ to each block $Q_l$ in $K_{-1}$.  

Set $K_{-2} = \bigcup_l \k_{z_2}(Q_l)$ and $T_{-1} = \overline{K_{-1} \setminus K_{-2}} = \bigcup_l \tau_{z_2}(Q_l)$.  Since $\mathsf{T}_{-1} = \bigcup_l \mathsf{t}_l$, the mapping $f|\mathsf{T}_{-1}: \mathsf{T}_{-1} \to T_{-1}$ is defined by gluing together homeomorphisms 
\[f|\mathsf{t}_l = \sigma_{z_1}^l \circ \Theta^{\mathcal{F}(l)}_{z_2} \circ (\zeta_{z_1}^l)^{-1}: \mathsf{t}_l \to \tau_{z_2}(Q_l)\]
where $\zeta^l_{z_1} : (\mathsf{Q},\mathsf{e}) \to (\mathsf{Q}_l,\mathsf{w}\cap \mathsf{Q}_l)$ is the similarity defined in Section \ref{sec:BLmaps}.

The union $W_{-1}$ of marked flag-edges $\varepsilon_l$ is a flag-path along $\k_{z_2}(Q_I)$ going from $\{(\frac{M-2}{M^2}\mathcal{C}^j)\times\{0\} : \mathcal{C}^j \in \mathcal{F}_0\}$ to $\{(\frac{M-2}{M^2}\mathcal{C}^j)\times\{1\} : \mathcal{C}^j \in \mathcal{F}_0\}$, and the restrictions of $f|\mathsf{t}_l$ to the entrance and to the exit of $\mathsf{t}_l$ are identical modulo isometries for all $l$. Hence, we conclude that the gluing, therefore the homeomorphism $f|\mathsf{T}_{-1}$, is well-defined. We now have the extension $f: \mathsf{T}_0 \cup \mathsf{T}_{-1} \to T_0 \cup T_{-1}$.  

For the next step, the index $l$ in the previous step is replaced by $l_1$. 

Fix $l_1 \in \{1,\dots, M^{1+z_1}\}$. Associated to each of the $M^{1+z_2}$ blocks $Q_{l_1, l_2}$ ($1 \leq l_2\leq M^{1+z_2}$) in the core $\k_{l_1} = \k_{z_2}(Q_{l_1})$, the process of defining $f|\mathsf{t}_{l_1}$ has uniquely defined a core $\k_{l_1, l_2}=\k_{z_3}(Q_{l_1,l_2})$, a tube $\tau_{l_1, l_2} = \tau_{z_3}(Q_{l_1,l_2})$, a marked flag-edge $\varepsilon_{l_1,l_2}$, a block $Q(l_1, l_2) \in \{Q_I, Q_J\}$, an $(N-1)$-flag $\mathcal{F}(l_1, l_2)$ of $\mathfrak{C}^{N+1}$, and a similarity mapping 
\[ \sigma_{z_1,z_2}^{l_1,l_2} : (Q(l_1,l_2), e_{\mathcal{F}(l_1,l_2)}) \to (Q_{l_1,l_2}, \varepsilon_{l_1,l_2}).\]
Similarly, we define for each $l_2=1,\dots,M^{1+z_2}$ a similarity mapping
\[ \zeta_{z_1,z_2}^{l_1,l_2} : (\mathsf{Q}, \mathsf{e}) \to (\mathsf{Q}_{l_1,l_2}, \mathsf{w}\cap \mathsf{Q}_{l_1,l_2}).\]  

The union $W_{-2}$ of these $M^{2 + z_1 + z_2}$ marked flag-edges is a flag-path along $K_{-2}$ from $\{(\frac{M-2}{M^3}\mathcal{C}^j)\times\{0\} : \mathcal{C}^j \in \mathcal{F}_0\}$ to $\{(\frac{M-2}{M^3}\mathcal{C}^j)\times\{1\} : \mathcal{C}^j \in \mathcal{F}_0\}$, and the union $K_{-3}$ of the cores of these $M^{2+z_1+z_2}$ new blocks is a topological $(N+2)$-cube. Set $T_{-2} = \overline{K_{-2} \setminus K_{-3}}$.  We now extend $f: \mathsf{T}_0 \cup \mathsf{T}_{-1} \cup \mathsf{T}_{-2} \to T_0 \cup T_{-1} \cup T_{-2}$ by gluing together the homeomorphisms 
\[f|\mathsf{t}_{l_1,l_2} = \sigma_{z_1,z_2}^{l_1,l_2} \circ \Theta_{z_3}^{\mathcal{F}(l_1,l_2)} \circ (\zeta_{z_1,z_2}^{l_1,l_2})^{-1}|\mathsf{t}_{l_1,l_2} \to \tau_{l_1, l_2}.\]

Continuing this process inductively in a self-similar manner, we obtain a homeomorphism $f : \mathsf{Q}\setminus (\{0\}\times [0,1]) \to Q_I \setminus \gamma$, where $\gamma$ is the snowflake open curve $\gamma = \bigcap_{k=1}^\infty K_{-k}$. 

\begin{lem}\label{lem:qsimil}
There exists $C > 1$ depending only on $N,n$ such that $M^{- \a k}f$ is $C$-bi-Lipschitz on each of the $M^{k+z_1+\dots+z_k}$ tubes in $\mathsf{T}_{-k}$.
\end{lem}

\begin{proof} The scaling factor of each $\zeta_{z_1,\dots,z_k}^{l_1,\dots,l_k}$ is $M^{-k-z_1-\dots-z_k}$ and the scaling factor of each $\sigma_{z_1,\dots,z_k}^{l_1,\dots,l_{k}}$ is $\frac{M-2}{M}M^{-k}$. Moreover, only a finite number of different bi-Lipschitz mappings $\Theta_{z}^{\mathcal{F}}$ have been used in the definition of $f$. Therefore, by Lemma \ref{lem:sum}, $M^{-\a k}f$ is $C$-bi-Lipschitz on each of the $M^{k+z_1+\dots+z_k}$ tubes in $\mathsf{T}_{-k}$, for some constant $C > 1$ depending on $M$, $p_{N,n}$, and the bi-Lipschitz constants of the maps $\Theta^{\mathcal{F}}_{0},\Theta^{\mathcal{F}}_{p_{N,n}}$; thus $C$ depends only on $N,n$.
\end{proof} 

Hence, the mapping $f: \mathsf{Q}\setminus(\{0\}\times [0,1]) \to Q_I \setminus \gamma$ is $K$-quasiconformal for some $K$ depending only on $N,n$. By a theorem of V\"ais\"al\"a for removable singularities \cite[Theorem 35.1]{Vais1}, $f$ can be extended to a $K$-quasiconformal mapping from $\mathsf{Q}$ onto $Q_I$.

\begin{rem}
Note the following self-similar property on $I$-blocks: whenever $Q_{l_1,\dots,l_{a+b}}$ is an $I$-block of $K_{-a-b}$ then
\begin{equation}\label{eq:selfsimil}
f|\mathsf{t}_{l_1,\dots,l_{a+b}} = \sigma_{z_1,\dots,z_{a+b}}^{l_1,\dots,l_{a+b}} \circ f|\mathsf{t} \circ (\zeta_{z_1,\dots z_{a+b}}^{l_1,\dots,l_{a+b}})^{-1}|\mathsf{t}_{l_1,\dots,l_{a+b}}.
\end{equation}
\end{rem}

In particular, the periodicity of $\{z_{k}\}$ with period $a+b$ implies the periodicity of $f$ (up to similarities) to tubes $\mathsf{t}_{l_1,\dots,l_{k}}$ and $\mathsf{t}_{l_1,\dots,l_{k+a+b}}$ when $Q_{l_1,\dots,l_{k}}$, $Q_{l_1,\dots,l_{k+a+b}}$ are $I$-blocks.

Finally, note that the snowflake curve $\gamma = \bigcup_{k=1}^{\infty} K_{-k}$ is the image of the line segment $\{0\}\times[0,1]$ under $f$. 

\subsection{Quasiconformal extension to $\R^{N+2}$}\label{sec:QCext}
We now extend the mapping $f: \mathsf{Q} \to Q_I$ to a quasiconformal homeomorphism of $\mathbb{R}^{N+2}$ by backward iteration. 

Fix an $I$-block $Q_{l_1,\dots,l_{a+b}}$ in some core $\k_{l_1,\dots,l_{a+b}}$ of $Q_I$ with $l_i \neq  1, M^{1+z_i}$. Such a block exists by the first property of the path $J_I$ in Section \ref{sec:cores}. 

Let $\zeta = \zeta_{z_{1},\dots,z_{a+b}}^{l_1,\dots,l_{a+b}}$ be the similarity in $\mathbb{R}^{N+2}$ that maps $(\mathsf{Q}, \mathsf{e})$ to $(\mathsf{Q}_{l_1,\dots,l_{a+b}}, \mathsf{w} \cap \mathsf{Q}_{l_1,\dots,l_{a+b}})$, and $\sigma = \sigma_{z_1,\dots,z_{a+b}}^{l_1,\dots,l_{a+b}}$ be the similarity in $\mathbb{R}^{N+2}$ that maps $(Q_I, e_{\mathcal{F}_0})$ to $(Q_{l_1,\dots,l_{a+b}}, w_{\mathcal{F}_0} \cap Q_{l_1,\dots,l_{a+b}})$ as in Section \ref{sec:qcinQ}. Note that $\zeta$ has a scaling factor $M^{-(a+b)(1+\a)}$ and $\sigma$ has a scaling factor $\frac{M-2}{M}M^{-(a+b)}$.  

Because $l_i \neq 1, M^{1+z_i}$, the space $\mathbb{R}^{N+2}$ is the union of an increasing sequence of $I$-blocks and regular blocks 
\[ \R^{N+2} = \bigcup_{k \geq 0} \sigma^{-k} Q_I = \bigcup_{k \geq 0} \zeta^{-k} \mathsf{Q}.\]
If $l_i = 1$ for all $i=1,\dots,a+b$ or $l_i = M^{1+z_i}$ for all $i=1,\dots,a+b$ then these unions would be proper subsets of $\R^{N+2}$.  

Define homeomorphisms $F^{(k)}: \zeta^{-k}\mathsf{Q} \to \sigma^{-k}Q_I$, $k \geq 0$, by 
\begin{equation}\label{eq:selfsimil2}
F^{(k)} = \sigma^{-k} \circ f \circ \zeta^k.
\end{equation}
The self similar property \eqref{eq:selfsimil} implies that $f \circ \zeta|\mathsf{Q} = \sigma \circ f$.  Therefore, $F^{(k)}|\mathsf{Q} = \sigma^{-k} \circ f \circ \zeta^k|\mathsf{Q} = f$ for all $k \geq 0$, and $F^{(k')}|\zeta^{-k} \mathsf{Q} = F^{(k)}$ for all $k' \geq k \geq 0$. Thus, the mapping $F_{\a} = \lim_{k \to \infty} F^{(k)}: \R^{N+2} \to \R^{N+2}$ is well-defined. Moreover, since all mappings $F^{(k)}$ are $K$-quasiconformal, $F_{\a}$ is $K$-quasiconformal and therefore, $F_{\a}$ is $\eta$-quasisymmetric for some $\eta$ depending only on $N,n$.
 
\begin{rem}\label{rem:distinction}
The backward iteration depends on the fact that $\a$ is rational. In fact, for any real number $\a \in [N, p_{N,n}]$, the arguments of Lemma \ref{lem:sum} can be used to find a sequence $(z_k)_{k\geq 0}$ having terms in $\{0,p_{N,n}\}$ such that $|z_1+\cdots+z_k - k\a| \leq p_{N,n}$ for all $k\geq 0$. Therefore, a quasiconformal map $f: \mathsf{Q} \to Q_I$ can be constructed as in Section \ref{sec:qcinQ}. However, if $\a$ is irrational the sequence $(z_k)$ is not periodic and the backward iteration cannot be used to extend this map in all $\R^{N+2}$.
\end{rem}

We show now that the quasisymmetric mapping $F_{\a}$ satisfies (\ref{eq:quasisim2}).

\begin{proof}[{Proof of Proposition \ref{prop:qsimilarity}}]
By the self similar property (\ref{eq:selfsimil2}) and the scaling factors of $\zeta,\sigma$, it is enough to show (\ref{eq:quasisim2}) only for $x=(x',x'') \in \mathsf{Q}$ with $|x'|\neq 0$. Suppose that $x \in t_{l_1\cdots l_k}$. Then, $B(x,\frac{1}{2}|x'|)$ intersects at most $m$ annulus tubes $t_{l_1\cdots l_{k'}}$ for some $m$ depending only on $M$, thus on $N,n$. Since $|x'| \simeq M^{-k-z_1-\cdots-z_k} \simeq M^{-k(\a+1)}$, we deduce (\ref{eq:quasisim2}) by Lemma \ref{lem:qsimil}. 
\end{proof}

\section{Proof of Theorem \ref{thm:main}}\label{sec:proofofmainthm}

In this section we give the proof of Theorem \ref{thm:main}. Using Proposition \ref{prop:qsimilarity}, we first show in Section \ref{sec:rationalproof} the theorem when $\a\geq 0$ is a rational number and then, in Section \ref{sec:irrationalproof}, we prove the theorem for all real numbers $\a\geq 0$ applying an Arzel\`a-Ascoli limiting argument.

Assuming Theorem \ref{thm:main}, the proof of Corollary \ref{cor:ext} is as follows.

\begin{proof}[{Proof of Corollary \ref{cor:ext}}]
Suppose that $\a\in[0,1)$ and $g$ is a bi-Lipschitz embedding of the singular line $\G = \{x_1=0\}\subset \mathbb{G}_{\a}$ into $\R^2$. We show that $g$ extends to a bi-Lipschitz embedding of $\mathbb{G}_{\a}$ onto $\R^2$.

Let $f \colon \mathbb{G}_{\a} \to \R^2$ be the bi-Lipschitz mapping of Theorem \ref{thm:main}. Then, $g(\G)$ and $f(\G)$ are quasilines in $\R^2$ and $g\circ f^{-1}$ is a bi-Lipschitz homeomorphism between these quasilines. Consider an $\eta$-quasisymmetric mapping $h : \mathbb{R} \to f(\G)$. By the Beurling-Ahlfors quasiconformal extension \cite{BerAhl}, there exists a $K$-quasiconformal extension $F : \R^2 \to \R^2$ of $h$, with $K$ depending only on $\eta$ that satisfies
\[ \diam{F(I)} \simeq |DF(x)| \diam{I}\]
for every arc $I\subset \mathbb{R}\times\{0\}$ and every point $x\in\R^2$ for which $\dist(x,I)\simeq |I|$. Here the ratio constants depend only on $\eta$. Similarly, there exists a quasiconformal mapping $G : \R^2 \to \R^2$ extending $g\circ f^{-1} \circ h$ satisfying the properties of $F$. 

We claim that $F = G \circ F^{-1}\circ f$ is bi-Lipschitz extension of $g$. Indeed, for any point $x\in \R^2$ we have $|DF(x)|/|DG(x)| \simeq \diam{F(I)}/\diam{G(I)}$ for some suitable $I\subset \mathbb{R}\times\{0\}$. Since $g\circ f^{-1}$ is bi-Lipschitz, the last ratio is comparable to 1. Therefore, $|DF(x)|\simeq|DG(x)|$ and $G \circ F^{-1}$ is bi-Lipschitz.
\end{proof}

\subsection{Proof of Theorem \ref{thm:main} when $\a$ is rational}\label{sec:rationalproof}

We first recall two basic properties of the generalized Grushin metric. 

The dilation property states that for any $\a \geq 0$ and any $\d >0$,
\[d_{\mathbb{G}_{\a}}((\d x_1,\d^{1+\a}x_2),(\d y_1,\d^{1+\a}y_2)) = \d d_{\mathbb{G}_{\a}}((x_1,x_2),(y_1,y_2)).\]
This can be found in \cite{Bellaiche} for the case $\a = 1$, but it applies equally to the case of arbitrary $\a \geq 0$. 

Given $x = (x_1,x_2), y = (y_1,y_2) \in \mathbb{G}_{\a}$ we define the \emph{Grushin quasidistance}
\begin{equation*} \label{eq:grushinQuasimetric}
d_{\a}(x,y) = |x_1 - y_1| + \min\left\{ |x_2-y_2|^{\frac{1}{1+\a}}, \frac{|x_2-y_2|}{|x_1|^{\a}}\right\}.
\end{equation*}
It is well-known that the quasidistance $d_{\a}(x,y)$ is comparable to $d_{\mathbb{G}_{\a}}(x,y)$; see e.g. \cite[Theorem 2.6]{FrLa}. In fact, the following result holds true.
\begin{lem}\label{lem:grushinQuasimetric}
For each $m\geq 0$ there exists $C(m)>1$ such that for all $\a \in [0,m]$ and all $x,y \in \mathbb{G}_{\a}$
\[ C(m)^{-1}d_{\a}(x,y) \leq d_{\mathbb{G}_{\a}}(x,y) \leq C(m)d_{\a}(x,y)\]
\end{lem}
The proof of Lemma \ref{lem:grushinQuasimetric} is identical to that of Lemma 2 in \cite{Ack}.
The next lemma is a simple application of the Mean Value Theorem and its proof is left to the reader.

\begin{lem}\label{lem:MVT}
For all $m\geq 0$ there exists $c(m)>1$ such that for all $\a\in [0,m]$ and $x, y\in\R$ with $|x|\geq|y|$,
\[ c(m)^{-1} |x|^{\a}|x-y| \leq |x|x|^{\a} - y|y|^{\a}| \leq c(m)|x|^{\a}|x-y|.\]
\end{lem} 


For each number $\a\in [N,N+\frac{n-1}{n}]$ define $H_{\a}: \mathbb{G}_{\a} \to \R\times\{0\}\times\R \subset \R^{N+2}$ to be the mapping
\[ H_{\a}(x_1, x_2) = (x_1|x_1|^{\a},0,\dots,0, x_2).\]
It is known that $H_{\a}$ is an $\eta'$-quasisymmetric mapping with $\eta'$ depending only on $N,n$; see e.g. \cite[Theorem 2]{Ack}.

We are now ready to prove Theorem \ref{thm:main} for rational $\a\geq 0$. The argument in this case is analogous to those of \cite[Theorem 1.1]{Wu} and \cite[Theorem 5.1]{Wu2}.

\begin{prop}\label{prop:rational}
For all integers $N\geq 0$ and $n\geq 1$, there exists $L>1$ depending only on $N,n$ such that, for each rational $\a \in [N,N+\frac{n-1}{n}]$, there exists an $L$-bi-Lipschitz homeomorphism of $\mathbb{G}_{\a}$ onto a $2$-dimensional quasiplane $\mathcal{P}_{\a} \subset \R^{N+2}$.
\end{prop}

\begin{proof}
Fix a rational $\a\in [N,N+\frac{n-1}n]$ and let $\lambda$ and $F_{\a}: \R^{N+2} \to \R^{N+2}$ be the constant and $\eta$-quasisymmetric map, respectively, of Proposition \ref{prop:qsimilarity} with $\lambda$ and $\eta$ depending only on $N,n$. The composition $F_{\a} \circ H_{\a}$ is a homeomorphism from $\mathbb{G}_{\a}$ onto the quasiplane $\mathcal{P}_{\a} = F_{\a}(\R\times\{0\}\times\R)$. We show that $F_{\a} \circ H_{\a}$ is $L$-bi-Lipschitz with $L$ depending only on $\lambda$, the quasisymmetric data $\eta,\eta'$ of $F_{\a},H_{\a}$, respectively, the constant $C(N+\frac{n-1}{n})$ of Lemma \ref{lem:grushinQuasimetric} and the constant $c(N+\frac{n-1}{n})$ of Lemma \ref{lem:MVT}; thus $L$ depends only on $N,n$. The comparison constants below depend at most on $N,n$. 

Let $x = (x_1, x_2)$, $y = (y_1, y_2)$ be points in $\mathbb{G}_{\a}$ and assume that $|x_1| \geq |y_1|$. 
The proof splits into four cases.  

\emph{Case I.} $|x_1| > 0$, $|x_1 - y_1| \leq |x_1|/4$, and $|x_2 - y_2| \leq |x_1|^{1+\a}/2$. Then, $|x_1| \simeq |y_1|$ and the Grushin distance satisfies $d_{\mathbb{G}_{\a}}(x,y) \simeq |x_1 - y_1| + |x_1|^{-\a}|x_2-y_2|$. Moreover, by Lemma \ref{lem:MVT}, $|H_{\a}(x)-H_{\a}(y)| \simeq |x_1|^{\a}|x_1-y_1| + |x_2-y_2|$. 

If $H_{\a}(y) \in B(H_{\a}(x),\frac{1}{2}|x_1|^{1+\a})$ then Proposition \ref{prop:qsimilarity} yields $|F_{\a} \circ H_{\a}(x) - F_{\a} \circ H_{\a}(y)| \simeq |x_1-y_1| + |x_1|^{-\a}|x_2 - y_2| \simeq d_{\mathbb{G}_{\a}}(x,y)$. 

Otherwise, $|H_{\a}(x)-H_{\a}(y)| \simeq |x_1|^{1+\a}$. Let $z\in\mathbb{G}_{\a}$ such that $|H_{\a}(x)-H_{\a}(z)| = |x_1|^{1+\a}/2$. Then the quasisymmetry of $F_{\a}$ and $H_{\a}$ implies $|F_{\a} \circ H_{\a}(x) - F_{\a} \circ H_{\a}(y)| \simeq |F_{\a} \circ H_{\a}(x) - F_{\a} \circ H_{\a}(z)| \simeq d_{\mathbb{G}_{\a}}(x,z) \simeq d_{\mathbb{G}_{\a}}(x,y)$.

\emph{Case II.} $|x_1| > 0$, $|x_1 - y_1| \geq |x_1|/4$, and $|x_2 - y_2| \leq |x_1|^{1+\a}/2$. Then, $d_{\mathbb{G}_{\a}}(x,y) \simeq |x_1 - y_1| \simeq |x_1|$ and, by Lemma \ref{lem:MVT}, $|H_{\a}(x)-H_{\a}(y)| \simeq |x_1|^{1+\a}$. Similar to the second part of Case I, $|F_{\a} \circ H_{\a}(x) - F_{\a} \circ H_{\a}(y)| \simeq d_{\mathbb{G}_{\a}}(x,y)$.

\emph{Case III.} $|x_1| > 0$ and $|x_2 - y_2| \geq |x_1|^{1+\a}/2$. Then, $d_{\mathbb{G}_{\a}}(x,y) \simeq |x_1 - y_1| + |x_2 - y_2|^{1/(1+\a)} \simeq |x_2 - y_2|^{1/(1+\a)}$. 
By Corollary \ref{cor:qsimilarity}, $|F_{\a} \circ H_{\a}(x) - F_{\a} \circ H_{\a}(y)| \simeq |H_{\a}(x) - H_{\a}(y)|^{\frac{1}{1+\a}} \simeq |x_2-y_2|^{\frac{1}{1+\a}} \simeq d_{\mathbb{G}_{\a}}(x,y)$.

\emph{Case IV.} $x_1 = 0$. Then, $|F_{\a} \circ H_{\a}(x) - F_{\a} \circ H_{\a}(y)| \simeq d_{\mathbb{G}_{\a}}(x,y)$ by taking limits in Case III.
\end{proof}

\subsection{Proof of Theorem \ref{thm:main} when $\a$ is irrational}\label{sec:irrationalproof}

The following lemma deals with the bi-Lipschitz embeddability of $\mathbb{G}_{\a}$ into $\R^{[\a]+2}$ for all real $\a\geq 0$.

\begin{lem}\label{lem:BLembed}
For all integers $N\geq 0$ and $n\geq 1$, there exists $L>1$ depending only on $N,n$ such that for all $\a \in [N,N+\frac{n-1}{n}]$ there exists an $L$-bi-Lipschitz embedding $f_{\a} : \mathbb{G}_{\a} \to \R^{N+2}$.
\end{lem}

\begin{proof}
Fix a number $\a \in [N,N+\frac{n-1}{n}]$ and let $(q_k)_{k\in\N}$ be a sequence of rational numbers in $[N,N+\frac{n-1}{n}]$ converging to $\a$. Note that $\lim_{k\to\infty}d_{\mathbb{G}_{q_k}}(x,y) = d_{\mathbb{G}_{\a}}(x,y)$ for each $x,y\in\R^{2}$. By Proposition \ref{prop:rational}, there exists $L>1$ depending only on $N,n$ such that, for each $q_k$, there is an $L$-bi-Lipschitz map $f_{q_k} : \mathbb{G}_{q_k} \to \R^{N+2}$. It is clear by their construction that each $f_{q_k}$ maps $(0,0)$ to $(0,\dots,0) \in\R^{N+2}$.

Let $\mathscr{A}= \{a_1,a_2,\dots\}$ be a countable dense set in $(\mathbb{G}_{\a},d_{\mathbb{G}_\a})$. Note that, for each $i\in\N$, $|f_{q_k}(a_i)| \leq L d_{\mathbb{G}_{q_k}}(a_i,(0,0))$. Hence, for each $i\in\N$, the sequence $(f_{q_k}(a_i))_{k\in\N}$ is bounded. Define, for each $i\in\N$, a subsequence of $(f_{q_k})_{k\in\N}$ as follows. Set $(f^0_k)_{k\in\N} = (f_{q_k})_{k\in\N}$ and for each $i \in\N$ let $(f^i_k)_{k\in\N}$ be a subsequence of $(f^{i-1}_k)_{k\in\N}$ so that $(f^i_k(a_i))_{k\in\N}$ converges. Then, for each $a_i\in \mathscr{A}$, the sequence $(f^k_k(a_i))_{k\in\N}$ converges. Set $f(a_i) = \lim_{k\to\infty} f^k_k(a_i)$. 

We claim that $f: (\mathscr{A},d_{\mathbb{G}_{\a}}) \to \R^{N+2}$ is $L$-bi-Lipschitz. Let $z_1,z_2 \in \mathscr{A}$ and $\e>0$. Choose $k\in\N$ big enough so that
\begin{equation}\label{eq:BL1}
|f^k_k(z_i)-f(z_i)| \leq \frac{\e}{3} \quad \text{ for each } i=1,2
\end{equation}
and if $f^k_k = f_{q(k)}$ for some $q(k)\in\{q_1,q_2,\dots\}$ then 
\begin{equation}\label{eq:BL2}
 |d_{\mathbb{G}_{q(k)}}(z_1,z_2) - d_{\mathbb{G}_{\a}}(z_1,z_2)| \leq \frac{\e}{3L}.
\end{equation}
Combining (\ref{eq:BL1}) and (\ref{eq:BL2}) we have that
\begin{equation*}
|f(z_1)-f(z_2)|\leq Ld_{\mathbb{G}_{\a}}(z_1,z_2) + \e.
\end{equation*}
Similarly, $|f(z_1)-f(z_2)| \geq \frac{1}{L}d_{\mathbb{G}_{\a}}(z_1,z_2) - \e$. Since $\e$ is arbitrary, the claim follows.

Using the density of $\mathscr{A}$ in $\mathbb{G}_{\a}$, the mapping $f$ can be extended to all $\mathbb{G}_{\a}$ uniquely. It remains to show that $f: \mathbb{G}_{\a} \to \R^{N+2}$ is $L$-bi-Lipschitz. Let $x_1,x_2 \in \mathbb{G}_{\a}$ and $\e>0$. Find $z_1,z_2\in \mathscr{A}$ such that for each $i=1,2$, $d_{\mathbb{G}_{\a}}(x_i,z_i) < \frac{\e}{4L}$ and $|f(x_i)-f(z_i)| < \frac{\e}{4}$. Then,
\begin{align*}
|f(x_1)-f(x_2)| \leq L d_{\mathbb{G}_{\a}}(z_1,z_2) + \frac{\e}{2}
                 \leq L d_{\mathbb{G}_{\a}}(x_1,x_2) + \e.
\end{align*}
Similarly, $|f(x_1)-f(x_2)| \geq \frac{1}{L}d_{\mathbb{G}_{\a}}(x_1,x_2) - \e$. Since $\e$ is arbitrary, $f$ is $L$-bi-Lipschitz.
\end{proof}

We now prove Theorem \ref{thm:main}.

\begin{proof}[{Proof of Theorem \ref{thm:main}}]
Let $\a\in[N,N+\frac{n-1}{n}]$ and $(q_k)$ be a sequence of rationals in $[N,N+\frac{n-1}{n}]$ converging to $\a$ such that the $L$-bi-Lipschitz maps $f_{q_k} = F_{q_k} \circ H_{q_k}$ converge to an $L$-bi-Lipschitz map $f_{\a}: \mathbb{G}_{\a} \to \R^{N+2}$ as in the proof of Lemma \ref{lem:BLembed}. Here $F_{q_k}:\R^{N+2} \to \R^{N+2}$ is the quasisymmetric mapping of Proposition \ref{prop:qsimilarity}, $H_{q_k}(x,y) = (x|x|^{q_k},0,\dots,0,y)$ is the quasisymmetric mapping of $\mathbb{G}_{q_k}$ onto $\R^2$ and $L$ depends only on $N,n$. Note that the mappings $H_{q_k}$ converge pointwise to the mapping $H_{\a} = (x|x|^{\a},0,\dots,0,y)$ and that the mappings $F_{q_k}$ fix the origin of $\R^{N+2}$ and the vector $(0,\dots,0,1)$. By \cite[Corollary 10.30]{Heinonen}, passing to a subsequence, we may assume that $F_{q_k}$ converges to a quasisymmetric mapping $F_{\a}$. Then, $f_{\a} = H_{\a}\circ F_{\a}$, and the image of $f_{\a}$ is $F_{\a}(\R\times\{0\}\times\R)$ which is a $2$-dimensional quasiplane in $\R^{N+2}$.
\end{proof}

\section{Appendix} \label{sec:paths}
This section gives the construction of the paths $J_I(N,n),J_L(N,n)$ used in Section \ref{sec:cores} and the proof of Lemma \ref{lem:BLext}. In Section \ref{sec:firstpaths}, we construct for each integer $N\geq 0$ and each integer $M = 4k+5\geq 9$ paths $\mathcal{J}^{N}_I(M)$, $\mathcal{J}^{N}_I(M)$ which serve as a base for the construction of paths $J_I(N,n),J_L(N,n)$ in Section \ref{sec:pathsN}. Then, in Section \ref{sec:exten_proof} we show Lemma \ref{lem:BLext}.

\subsection{Auxiliary paths}\label{sec:firstpaths}
Let $N\geq 0$ and $M= 4k+5 \geq 9$ be integers. The paths $\mathcal{J}^{N}_I(M), \mathcal{J}^{N}_I(M)$ are defined by induction on $N$.

For an integer $M= 4k+5 \geq 9$ let $\mathcal{J}^{0}_I(M)$ be the segment $I\subset\R^2$ which we divide into $M$ disjoint $I$-segments $\ell_m$ of length $1/M$. Similarly, let $\mathcal{J}^{0}_L(M)$ be the segment $L\subset\R^2$ which we divide into $M$ disjoint $I$- and $L$-segments $\ell_m$ of length $1/M$ where $\ell_{\frac{M-1}{2}}$ is an $L$-segment and the rest are $I$-segments.

To obtain $\mathcal{J}^{1}_I(M)$, replace each pair of $I$-segments $\ell_{m}\cup\ell_{m+1}$, where $m \in \{2, 4, \dots, \frac{M-5}{2}, \frac{M+5}{2},\dots, M-4, M-2\}$, by a swath containing $\frac{M-1}{2}$ $I$- and $L$-segments of length $1/M$ running in the negative $x_1$-direction; see Figure \ref{figure:IPath} for a schematic representation. Precisely, $\mathcal{J}^{1}_I(M)$ contains $\frac{M-5}{2}$ swaths and each swath contains $4$ $L$-segments and $\frac{M-5}{2}$ pairs of consecutive $I$-segments. Here we make use of the fact that $M=4k+5$.

To obtain $\mathcal{J}^{2}_I(M)$ replace each of the $(M-5)^2/4$ pairs of consecutive $I$-segments in $\mathcal{J}^{2}_I(M)$ by a swath containing $\frac{M-1}{2}$ many $I$- and $L$-segments of length $1/M$ running in the positive $x_3$-direction; see Figure \ref{figure:IPath}. Note that $\mathcal{J}^{2}_I(M)$ contains $(M-5)^2/4$ new swaths, each containing $\frac{M-5}{2}$ pairs of consecutive $I$-segments.

Proceeding inductively, we obtain for each integer $N\geq 0$ and each integer $M = 4k+5\geq 9$ a path $\mathcal{J}^{N}_I(M)$. Denote by $(\#\mathcal{J}^{N}_I(M))$ the total number of $I$- and $L$-segments in $\mathcal{J}^{N}_I(M)$, and by $(\#\mathcal{J}^{N}_I(M))^*$ the total number of pairs of consecutive $I$-segments. Then, $(\#\mathcal{J}^{0}_I(M)) = M$, $(\#\mathcal{J}^{0}_I(M))^* = \frac{M-5}{2}$ and for $N\geq 1$
\begin{eqnarray*}
&(\#\mathcal{J}^{N}_I(M)) &= (\#\mathcal{J}^{N-1}_I(M)) + (\#\mathcal{J}^{N-1}_I(M))^*(M-3)\\
&(\#\mathcal{J}^{N}_I(M))^* &= (\#\mathcal{J}^{N-1}_I(M))^* \frac{M-5}{2}.
\end{eqnarray*} 
Therefore,
\[ (\#\mathcal{J}^{N}_I(M)) = M + (M-3)(M-5)\frac{(M-5)^{N}-2^N}{2^{N+1}(M-7)}\]
and
\[ (\#\mathcal{J}^{N}_I(M))^* = \frac{(M-5)^{N+1}}{2^{N+1}}.\]

The paths $\mathcal{J}^{N}_L(M)$ are constructed similarly; see Figure \ref{figure:IPath}.

\begin{figure}[h]
\centering 
{
\includegraphics[width=1.3in]{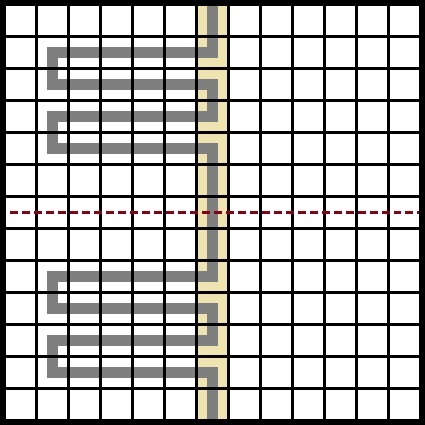}
}
\hspace{.1in}
{
\includegraphics[width=1.3in]{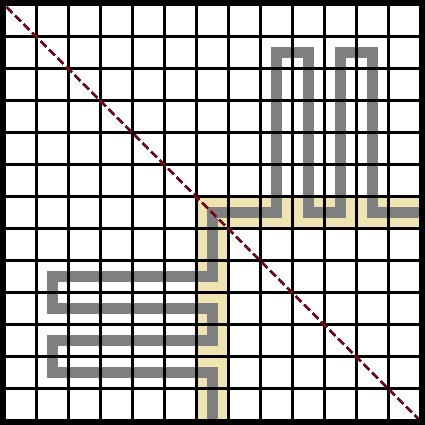}
}
\hspace{.1in}
{
\includegraphics[width=1.3in]{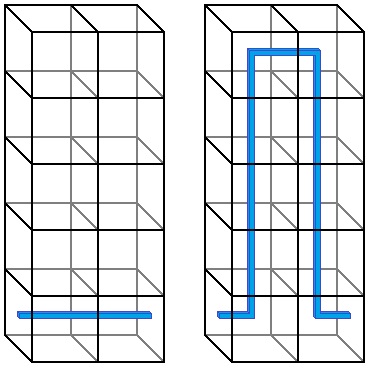}
}
\caption{The paths $\mathcal{J}^0_I(M)$, $\mathcal{J}^0_L(M)$ and a swath in the extra dimension.} \label{figure:IPath}
\end{figure}    

\subsection{Construction of the paths $J_I,J_L$} \label{sec:pathsN} 
Fix integers $N\geq 0$ and $n\geq 1$ and set $M = M_{N,n}=9^{n(N+2)}$. We first construct paths $\tilde{J}_I(N,n)$ and $\tilde{J}_L(N,n)$ as an extension of $\mathcal{J}^N_{I}$ and $\mathcal{J}^N_{L}$, respectively, in an extra dimension. The required paths $J_I(N,n)$ and $J_L(N,n)$ are obtained after applying a suitable rotation to $\tilde{J}_I(N,n)$ and $\tilde{J}_L(N,n)$ respectively.

We work first for $\tilde{J}_{I}(N,n)$. To construct $\tilde{J}_I(N,n)$ we use the path $\mathcal{J}^{N}_I(M)$ which contains $M' = (\#\mathcal{J}^{N}_I(M))^* = 2^{-N-1}(M-5)^{N+1}$ pairs of consecutive $I$-segments. Replace each pair of $I$-segments $\ell_m\cup \ell_{m+1}$ in $\mathcal{J}^{N}_I(M)$, $m=1,\dots,M'$, by a swath consisting of $2k_m+2$ many $I$- and $L$-segments, running in the positive $x_{N+2}$-direction. Here, $0 \leq k_m \leq \frac{M-3}2$ (if $k_m = 0$ then the swath contains only $\ell_m,\ell_{m+1}$ and if $k_m\geq 1$ then it contains $4$ $L$-segments and $2(k_m-1)$ $I$-segments). The resulting path is denoted by $\tilde{J}_I(N,n)$. Moreover we require that the swaths are chosen in such a way that $\tilde{J}_I(N,n)$ is symmetric with respect to the plane $x_{N+1} = 1/2$. Hence, for each $m\in \{1,\dotsm, M'\}$ there is $m' \in \{1,\dots,M'\}$, $m\neq m'$ such that $k_m = k_{m'}$.

The path $\tilde{J}_I(N,n)$ must consist of $M^{1+p_{N,n}} = M^{N+2-1/n}$ many $I$- and $L$-segments of length $1/M$. Thus, we require that
\[ 2(k_1 + k_2 +\cdots + k_{M'}) + 2M' + ((\#\mathcal{J}^{N}_I(M)) - 2M') = M^{N+2-1/n}\]
or equivalently
\begin{equation}\label{eq:paths}
k_1 + k_2 +\cdots + k_{M'} = \frac{M^{N+2-1/n} - (\#\mathcal{J}^{N}_I(M))}{2}.
\end{equation}
The symmetry of $\tilde{J}_I(N,n)$ implies that the left hand side of (\ref{eq:paths}) is even. Moreover, since $M$ is a multiple of $9$, the right hand side of (\ref{eq:paths}) is also even. Since each $k_m$ can take any integer value in $[0,M-3]$, the left hand side of (\ref{eq:paths}) can take any even integer value in $[0,2(M-3)M']$ and it is enough to show that
\[ 2(M-3)M' \geq \frac{M^{N+2-1/n} - (\#\mathcal{J}^{N}_I(M))}{2}.\]
Indeed, since $M = 9^{n(N+2)}$,
\[ 2(M-3)M' = \frac{(M-3)(M-5)^{N+1}}{2^N} \geq \left ( \frac{M-5}{2} \right )^{N+2} \geq M^{N+2-\frac{1}{n}}.\]

Properties (1)--(4) of Section \ref{sec:cores} are immediate. The proof of property (5) is almost identical to the proof of Lemma \ref{lem:BLext} in the following section. The path $\tilde{J}_{L}(N,n)$ is obtained similarly. In this case we require symmetry with respect to the plane $x_1+x_{N+1}= \frac{1}{2}$.

\subsection{Proof of Lemma \ref{lem:BLext}} \label{sec:exten_proof}

We show Lemma \ref{lem:BLext} for $z=p$ and $Q=Q_I$. Similar arguments apply when $Q=Q_L$. For the rest, $\mathcal{F} = \mathcal{F}_0$.

By Section \ref{sec:pathsN}, each $J_I(N,n)$ is constructed as a sequence of paths $I = J_1,J_2,\dots,J_{N+2} = J_I(N,n)$ where each $J_k$ lies in a $k$-dimensional subspace of $\R^{N+2}$ and $J_{k+1}$ is constructed by replacing pairs of $I$-segments $\ell_m \cup \ell_{m+1}$ of $J_k$ by swaths $\mathscr{S} = I_m\cup\mathscr{S}_m\cup I_{m+1}'$. Here, $I_m\subset\ell_m$, $I_{m+1}'\subset\ell_{m+1}$ are line segments and $\mathscr{S}_m$ is a polygonal arc perpendicular to the $k$-plane containing $J_k$. Associated to each $J_k$ we consider a core $\kappa_k = \mathcal{T}^{N+2}(J_k,\frac{M-2}{M^2})$. 

Each core $\kappa_k$ consists of $M_k$ many $I$- and $L$-blocks $Q_{k,m}$, $m=1,\dots,M_k$. Here $M_k = (\#\mathcal{J}^{k}_I(M))$, if $k=0,\dots,N+1$, and $M_{N+2}=M^{1+p_{N,n}}$ with $M=9^{n(N+2)}$. Similar to the path $J_k$, each core $\kappa_k$ is constructed by removing certain pairs of $I$-blocks from $\kappa_{k-1}$ and replacing these pairs by solid swaths $\mathcal{S} = \mathcal{T}^{N+2}(\mathscr{S},\frac{M-2}{M^2})$. Note that $\kappa_1 = \kappa_0(Q)$.

For each $k,m$ the side $\text{s}(Q_{k,m})$ has a fibration into $I$-segments (if $Q_{k,m}$ is an $I$-block) or $L$-segments (if $Q_{k,m}$ is an $L$-block) similar to the fibrations $\{I_x\},\{L_x\}$ of Section \ref{sec:BLmaps}. The fibrations of the sides of $Q_{k,m}$ induce a fibration of the side $\text{s}(\kappa_k) = \bigcup_{u} \Gamma_{k,u}$ where $\Gamma_{k,u}$ is a polygonal arc, $u\in \partial\mathfrak{C}^{N+1}$ and $\Gamma_{k,u}\cap \text{s}(Q_{k,m})$ is a fiber of $\text{s}(Q_{k,m})$. As with the paths $J_k$, each $\G_{k+1,u}$ is constructed by replacing certain line segments of $\G_{k,u}$ by fibers which lie on the new solid swaths of $\k_{k+1}$. Note that if $u$ is a vertex of $\mathfrak{C}^{N+1}$ then $\G_{k,u}$ is an edge of $\kappa_k$ and $\G_{N+2,u}$ is an element of the flag-path $w_{\mathcal{F}_0}$ of $ \kappa_p(Q_I)$.

For the construction of $\Theta_p^{\mathcal{F}_0}$ we first map $\tau_p(Q)$ onto $\tau_{0}(Q)$ and then we compose with $\Theta_0^{\mathcal{F}_0}$. 

\emph{Step 1: We map $(Q,\kappa_{N+2})$ onto $(Q,\kappa_{1})$}. We construct a bi-Lipschitz map in $Q$ which fixes $\partial Q$ and maps $\kappa_{N+2}$ onto $\kappa_{N+1}$ by compressing each solid swath onto the two $I$-blocks of $\kappa_{N+1}$ which it replaced. The map is defined in a neighbourhood of each solid swath.

For each solid swath $\mathcal{S} \subset \kappa_{N+2}$, consider a $(N+2)$-box $\widetilde{Q}(\mathcal{S}) \subset Q$ which contains $\mathcal{S}$ and satisfies the following properties,
\begin{enumerate}
\item each face of $\widetilde{Q}(\mathcal{S})$ is parallel to a coordinate $(N+1)$-hyperplane;  
\item $\widetilde{Q}(\mathcal{S}) \cap \kappa_{N+2} = \mathcal{S}$;
\item $\widetilde{Q}(\mathcal{S})$ and $\widetilde{Q}(\mathcal{S}')$ have disjoint interiors if $\mathcal{S} \neq \mathcal{S}'$.
\end{enumerate}
For each solid swath $\mathcal{S} \subset \kappa_{N+2}$ we construct a bi-Lipschitz isotopy $\Phi_{\mathcal{S}}: \widetilde{Q}(\mathcal{S}) \times I \to \widetilde{Q}(\mathcal{S})$ such that $\Phi_{\mathcal{S}}(\cdot, t)|\partial \widetilde{Q}(\mathcal{S}) = \text{id}$ for all $t \in [0,1]$, $\Phi_{\mathcal{S}}(\cdot, 0) = \text{id}$, and $\Phi_{\mathcal{S}}(\cdot, 1)|\mathcal{S}$ is a PL bi-Lipschitz map of $\mathcal{S}$ onto the two $I$-blocks of $\kappa_{N+1}$ that $\mathcal{S}$ replaced. By PL bi-Lipschitz isotopy, we mean that the induced mapping $g_{t_1t_2} = \Phi_{\mathcal{S}}(\Phi_{\mathcal{S}}^{-1}(\cdot , t_1),t_2)$ is piecewise linear and $(1+C|t_2-t_1|)$-bi-Lipschitz for some constant $C>0$ and all $t_1, t_2 \in [0,1]$. Note that $g_{t_1t_2}^{-1} = g_{t_2t_1}$.

Fix a solid swath $\mathcal{S} \subset \kappa_{N+2}$ and write $\widetilde{Q}(\mathcal{S}) = \widetilde{Q}$ and $\Phi_{\mathcal{S}} = \Phi$. Suppose that $\mathcal{S} = Q_1'\cup\cdots\cup Q_{2m}'$ where $Q_{i}'$ are blocks of $\kappa_{N+2}$ and that $\mathcal{S}$ has replaced two $I$-blocks $Q_1\cup Q_2$ of $\kappa_{N+1}$.

If $m=1$ then $Q_1=Q_1'$, $Q_2=Q_2'$ and $\Phi$ is the identity in $\widetilde{Q}$.

Suppose now that $m\geq 2$. We write $Q_i = \mathcal{T}^{N+2}(\ell_i,\frac{M-2}{M^2})$ and $Q_j' = \mathcal{T}^{N+2}(\ell_j',\frac{M-2}{M^2})$ for $i=1,2$ and $j=1,\dots,2m$ where $\ell_1,\ell_2$ are $I$-segments, $\ell_i'$ is an $L$-segment when $i=1,m,m+1,2m$ and an $I$-segment otherwise. Let $\widehat{\ell}$ be an $I$-segment of length $\frac{1}{10M}$ intersecting both $\ell_1$ and $\ell_2$. Define $\widehat{\ell}_1 = \ell_1\setminus \widehat{\ell}$, $\widehat{\ell}_{2m} = \ell_2\setminus \widehat{\ell}$ and $\{\widehat{\ell}_j\}_{j=1}^{2m-1}$ be a partition of $\widehat{\ell}$ into $I$-segments of length $\frac{1}{(2m-2)10M}$.

Let $\Phi: \partial(\widetilde{Q}\setminus \mathcal{S}) \times I \to \widetilde{Q}$ be a PL bi-Lipschitz isotopy on $\partial(\widetilde{Q}\setminus \mathcal{S})$ such that $\Phi(\cdot, t)|\partial \widetilde{Q} = \text{id}$ for all $t \in [0,1]$, $\Phi(\cdot, 0) = \text{id}$, and $\Phi(\cdot, 1)|\partial\mathcal{S}$ maps each $Q_{j}'$ onto $\widehat{Q}_j = \mathcal{T}(\widehat{\ell}_j, \frac{M-2}{M^2})$, while each $\G_{N+2,u}\cap Q_j'$ is mapped onto $\G_{N+1,u}\cap \widehat{Q}_j$ by arc-length parametrization. Let $\Sigma_t = \Phi(\partial(\widetilde{Q}\setminus \mathcal{S}),t)$. 

We use the following theorem of V\"ais\"al\"a on bi-Lipschitz extensions. 

\begin{thm}[{\cite[Corollary 5.20]{Vais:86}}]\label{thm:vais_ext}
Let $n\geq 2$ and $\Sigma \subset \R^n$ be a compact piecewise linear manifold of dimension $n$ or $n-1$ with or without boundary. Then, there exist $L',L>1$ depending on $\Sigma$, such that every $L$-bi-Lipschitz embedding $f: \Sigma \to \mathbb{R}^n$ extends to an $L'$-bi-Lipschitz map $F: \mathbb{R}^n \to \mathbb{R}^n$.    
\end{thm} 

By Theorem \ref{thm:vais_ext}, for each $t\in[0,1]$, there are constants $L_t, L_t' >1$ such that any $L_t$-bi-Lipschitz map $f: \Sigma_t \to \R^{N+2}$ has an $L_t'$-bi-Lipschitz extension $F: \R^{N+2} \to \R^{N+2}$. For all $t \in [0,1]$, there is an open interval $\D_t$ such that $1+C|s-t|<L_t$ for all $s \in \D_t$. Cover $[0,1]$ with finitely many intervals $\{\D_{t_j}\}_{j=1}^l$, where $0 = t_0 < t_1 < \cdots < t_l = 1$ and $\D_{t_{j-1}}\cap \D_{t_j} \neq \emptyset$. For each $j=1,\dots,l$ set $a_{2j} = t_j$ and $a_{2j-1} \in \D_{t_{j-1}}\cap\D_{t_j}$. Then, each $g_{a_{j}a_{j+1}}$ extends to a bi-Lipschitz map $G_{a_{j}a_{j+1}}: \R^{N+2} \to \R^{N+2}$. Hence, $G_{a_{2l-1}a_{2l}}\circ \cdots \circ G_{a_0a_1}$ is a bi-Lipschitz self-map of $\widetilde{Q}$. The respective bi-Lipschitz maps on each $\widetilde{Q}$ can be pasted together and the resulting map is still bi-Lipschitz. 

Similarly we use a bi-Lipschitz map in $Q$ that maps $\kappa_{N+1}$ onto $\kappa_{N}$ satisfying all the properties of the previous bi-Lipschitz map. Inductively, we obtain a bi-Lipschitz map
\[ \Theta' : (Q,\kappa_{N+2}) \to  (Q,\kappa_{1})\]
such that $\Theta'$ is identity on $\partial Q$, maps each $\G_{N+2,u}$ on $\G_{1,u}$ and every block $Q_{N+2,m}$ in the core $\kappa_{N+2}$ is mapped to a block $\mathcal{T}^{N+2}(\ell,\frac{M-2}{M^2})$ where $\ell = \ell(m)$ is a straight line segment lying on $J_{1}$. Note that $J_1$ and all fibers $\G_{1,u}$ of $\kappa_1$ are straight line segments isometric to each other.

\emph{Step 2: We straighten the images of $\G_{N+2,u}$.} Consider the line segments 
\[ \Gamma'_{N+2,u}(m) = \Theta' (\Gamma_{N+2,u}\cap Q_{N+2,m})\]
and let $\Gamma'_{N+2,u} = \bigcup_{m=1}^{M_{N+2}}\Gamma'_{N+2,u}(m)$. The family $\{\Gamma'_{N+2,u}\}_{u\in \mathfrak{C}^{N+1} }$ is a fibration of $\kappa_1 = \kappa_0(Q)$ and if $u$ is a vertex of $\mathfrak{C}^{N+1}$ then $\Gamma'_{N+2,u}$ is an edge of $\kappa_1$. Let $\Theta'' : Q \to Q$ be a bi-Lipschitz mapping which is identity on $\partial Q$ and linear on each $\Gamma'_{N+2,u}(m)$. Moreover, for all $u,m$, $\Theta''(\Gamma'_{N+2,u}(m))$ lies on $\Gamma'_{N+2,u}$ and its length is $1/M$. Define now $\Theta_z^{\mathcal{F}_0} = (\Theta')^{-1}\circ(\Theta'')^{-1}\circ\Theta_0^{\mathcal{F}_0}$ and the proof is complete.

\bibliographystyle{abbrv}
\bibliography{biblio}
\end{document}